\documentclass{amsart}
\usepackage{color}
\usepackage[utf8]{inputenc}
\usepackage{amscd}
\usepackage{amsmath}
\usepackage{bm}
\usepackage[english]{babel}
\usepackage{amsfonts}
\usepackage{graphicx}
\usepackage{mathtools}
\usepackage{vmargin}
\usepackage{theoremref}
\usepackage{amssymb}
\usepackage{commath}
\usepackage{cancel}
\usepackage{comment}
\usepackage{xcolor}
\definecolor{RED}{rgb}{1,0,0}
\usepackage{lineno}
\usepackage{enumitem}
\newtheorem{theorem}{Theorem}[section]

\newtheorem{definition}[theorem]{Definition}
\newtheorem{proposition}[theorem]{Proposition}

\newtheorem{example}[theorem]{Example}

\newtheorem{remark}[theorem]{Remark}
\newtheorem{assumption}[theorem]{Assumption}
\newcommand{\R}{\mathbb{R}}

\numberwithin{equation}{section}

\title[]{A Green Function Approach to Smooth Nonautonomous Topological Equivalence with Unbounded Nonlinearities under $(\mu,\nu)$--Dichotomies}
\author[]{}
\address{Universidad de Chile, Departamento de Matem\'aticas. Casilla 653, Santiago, Chile}

\author[]{Fernanda Torres}
\address{}
\email{fernanda.torres.t@ug.uchile.cl}
\thanks{This research has been supported by FONDECYT Regular 1240361}
\date{\today}
\begin{document}


\begin{abstract}

We study the smooth topological equivalence between a nonautonomous linear system on the positive half-line and a quasilinear perturbation whose nonlinear part is not assumed to be globally bounded with respect to the state variable. The linear equation is assumed to admit a $(\mu,\nu)$--dichotomy, and the construction is carried out through the Green operator associated with this dichotomy. The usual global boundedness of the perturbation is replaced by locally uniform Green-integrability conditions along the relevant linear and nonlinear solutions. Under a smallness condition involving the global Lipschitz constant of the perturbation and the Green operator, we first construct Palmer-type maps which give a continuous topological equivalence on $\mathbb R^+$. We then impose Green-integrability conditions on the successive variational equations and a further first-order smallness condition which guarantees the invertibility of the derivative of the inverse map. Under these assumptions, the equivalence is of class $C^r$. We also provide a class of examples for which the perturbation is unbounded with respect to the space variable and all the hypotheses can be verified directly.

\end{abstract}

\maketitle

\section{Introduction}

The Hartman--Grobman Theorem \cite[Theorem I]{Hartman1} constitutes a fundamental result in the study of the local behavior of nonlinear dynamical systems. It guarantees the existence of a local topological conjugacy between the solutions of a nonlinear system and those of its linearization around a hyperbolic equilibrium point, establishing that both systems share the same topological structure in a neighborhood of such equilibrium. The extension of these ideas to the global setting was initiated by Pugh \cite{Pugh}, who examined a particular case of the Linearization Theorem focused on linear systems with bounded and Lipschitz perturbations, thereby enabling the construction of a global homeomorphism.

\subsection{Nonautonomous linearization}

The first nonautonomous linearization result traces back to the seminal work of Palmer \cite{Palmer}, which constructs upon Pugh's global approach. Subsequently, considerable attention has been devoted to establishing whether such nonautonomous linearizations possess some degree of differentiability. This question has attracted growing interest over the last decade, giving rise to a substantial body of work 
addressing it from diverse perspectives. To the best of our knowledge, the problem was first tackled by Casta\~neda and Robledo in \cite{JDE}, who, under suitable integrability conditions and without appealing to spectral theory, proved that the linearization is of 
class $C^2$ when the linear part is a uniform contraction on $\mathbb{R}$. A related result was subsequently obtained by the same authors together with Monz\'on in \cite{Monzon1}, where they established that the linearization is of class $C^r (r \geq 1)$ when the linear part admits a general nonuniform contraction on $\mathbb{R}^+$.

A first result concerning the differentiability of the linearization in the presence of both contraction and expansion, within the framework of spectral theory, was established by Dragicevic \textit{et al.} in \cite{Dragicevic2}. In that work, the differentiability of the linearization is obtained under the assumption that the linear part admits a strong nonuniform exponential dichotomy, meaning that the 
linear system possesses the nonuniform exponential dichotomy property and its transition matrix satisfies a nonuniform bounded growth 
condition on $\mathbb{R}^+$. Subsequently, in \cite{CMR}, the authors introduced the notion of \textit{continuous topological equivalence}, 
which, broadly speaking, requires the topological equivalence to depend continuously on both the time and space variables. Without appealing to spectral theory, they proved that if the linear system is a uniform contraction, then the linearization is of class $C^r(r \geq 1)$ and its partial derivatives are continuous in both variables. More recently, in \cite{Jara}, Jara proved that if the linear system admits a general nonuniform dichotomy and the nonlinear perturbation satisfies certain restrictive conditions, then the linearization is of class $C^2$ on $\mathbb{R}^+$, again without the use of spectral theory.

The global linearization results in \cite{CMR,JDE,Jara,Palmer, Reinfelds, Shi}, which are based on the Green function approach, consider a linear system and a quasilinear perturbation whose nonlinear part is globally Lipschitz and bounded.
From a comparative 
perspective, linearization results considering unbounded nonlinearities are less studied; results in this line are obtained by Lin \cite{FLIN}, in the uniform case, and Huerta \cite{Huerta} in the nonuniform case.

A first result of differentiability of the topological equivalence, concerning unbounded nonlinearities within the Green's approach was obtained by Casta\~neda and Torres in \cite{CT}. In that work, the linear system is a uniform contraction on $\R^+$ and the perturbation is allowed to be unbounded with respect to the space variable.

The purpose of this work is to continue the program initiated in \cite{CT}, replacing the uniform contraction by a $(\mu,\nu)$--dichotomy. The growth of the perturbation is controlled by a function $\omega(t,r)$, and this growth is compensated by locally uniform integrability properties of the Green operator along the relevant linear and nonlinear solutions.

The main contribution is twofold. First, we show that the Palmer-type construction remains valid under a $(\mu,\nu)$--dichotomy when the perturbation is unbounded, provided the Green operator satisfies suitable integrability assumptions and its interaction with the Lipschitz constant is sufficiently small. Second, we prove that the resulting topological equivalence is of class $C^r$. For this purpose, instead of assuming boundedness of the nonlinear solutions, we introduce locally uniform Green-integrability conditions for the successive variational equations. We also isolate a first-order smallness condition which guarantees that the derivative of the inverse equivalence is nonsingular at every point.

The differentiability assumptions are formulated directly in terms of the Green operator and the derivatives of the nonlinear solutions with respect to the initial conditions. This permits unstable directions and unbounded nonlinear solutions, since the growth of the variational equations may be compensated by the time decay of the derivatives of the perturbation.

\subsection{Structure of this paper}

The structure of this article is the following. Section 2 contains the $(\mu,\nu)$--dichotomy framework and the locally uniform Green-integrability assumptions replacing boundedness of the perturbation. Section 3 constructs the auxiliary functions $z^*$ and $w^*$ and establishes the continuous topological equivalence on $\R^+\times\R^n$. Section 4 introduces the Green-integrability conditions for the variational equations and proves that the equivalence is of class $C^r$. Section 5 provides a class of examples with unbounded nonlinearities. We finish with some comments concerning the scope of the result.

\subsection{Notations}
Throughout this paper, $||\cdot||$ and $|\cdot|$ will denote matrix and vector norms respectively. The set $[0,+\infty)$ is denoted by $\mathbb{R}^{+}$ and the set of square 
$n\times n$ matrices with real coefficients is denoted by $\mathcal{M}_{n}$, while $I_{n}$ is the identity matrix.

\section{Preliminary results and contextualization}

In this work, we consider a nonautonomous linear system

\begin{equation}\label{lin}
    x'=A(t)x
\end{equation}

and its quasilinear perturbation

\begin{equation}\label{nolin}
    y'=A(t)y+f(t,y).
\end{equation}

We denote by $t\mapsto x(t,\tau,\xi)$ and $t\mapsto y(t,\tau,\eta)$ the solutions of (\ref{lin}) and (\ref{nolin}) that pass through $\xi$ and $\eta$ respectively at $t=\tau$. We also denote by $T(t,s)$ the transition matrix of (\ref{lin}) such that for $t=s$ is $I_n$. Moreover, $A:\mathbb{R}^+\to \mathcal{M}_{n}$ is continuous and uniformly bounded, that is, there exists $\alpha>0$ such that
    \begin{equation}\label{acotamiento-A}
     \sup_{s \in \mathbb{R}^+}\norm{A(s)}=\alpha < +\infty. 
    \end{equation}

    \vspace{0.3cm}

\begin{definition}\label{dicotomia}
The system \eqref{lin} admits a $(\mu,\nu)$--dichotomy if there exist a continuous family of invariant projections $P(t)$ on $\mathbb R^n$, positive constants $M$ and $\lambda$, and continuous functions $\mu,\nu:\mathbb R^+\to[1,+\infty)$ such that $\mu$ is strictly increasing, and
$$
\mu(0)=1,
\qquad
\lim_{t\to+\infty}\mu(t)=+\infty.
$$
Moreover, setting $Q(t)=I-P(t)$, the transition matrix $T(t,\tau)$ satisfies
\begin{equation*}
\|T(t,\tau)P(\tau)\|
\leq
M\nu(\tau)
\left[\frac{\mu(t)}{\mu(\tau)}\right]^{-\lambda},
\qquad t\geq\tau,
\end{equation*}
and
\begin{equation*}
\|T(t,\tau)Q(\tau)\|
\leq
M\nu(\tau)
\left[\frac{\mu(\tau)}{\mu(t)}\right]^{-\lambda},
\qquad 0\leq t\leq\tau.
\end{equation*}
\end{definition}


To describe the bounded solutions of the corresponding nonhomogeneous system, it is convenient to introduce the Green function associated with the dichotomy. This function combines the forward evolution along the stable directions and the backward evolution along the unstable directions determined by the invariant projections $P(t)$ and $Q(t)$.

\begin{definition}
The Green function associated with system \eqref{lin} and its $(\mu,\nu)$-dichotomy is the matrix-valued function
$$
\mathcal{G}:\mathbb{R}^+\times\mathbb{R}^+\to\mathcal{M}_n(\mathbb{R})
$$
defined by
$$
\mathcal{G}(t,s)=
\begin{cases}
T(t,s)P(s) & t\geq s\geq 0 \\
-T(t,s)Q(s), & 0\leq t<s,
\end{cases}
$$
where $Q(t)=I-P(t)$ is the complementary projection.
\end{definition}

\vspace{0.3cm}

\begin{assumption}\label{hyp-main}
The following conditions hold.
\begin{enumerate}[label=\textnormal{(P\arabic*)},leftmargin=1.4cm]

\item\label{H1} The linear system \eqref{lin} admits a $(\mu,\nu)$--dichotomy in the sense of Definition \ref{dicotomia}.

\item\label{H2} The function $f:\mathbb R^+\times\mathbb R^n\to\mathbb R^n$ is continuous and there exists $\gamma>0$ such that
$$
|f(t,y)-f(t,\overline y)|
\leq
\gamma|y-\overline y|,
\qquad
t\geq0,\quad y,\overline y\in\mathbb R^n.
$$

\item\label{H3} There exists a continuous function
$$
\omega:\mathbb R^+\times\mathbb R^+\to\mathbb R^+
$$
which is nondecreasing in the second variable and satisfies
$$
|f(t,y)|\leq\omega(t,|y|),
\qquad
t\geq0,\quad y\in\mathbb R^n.
$$

\item\label{H4} The Green operator satisfies
\begin{equation*}
C:=
\sup_{t\geq0}
\int_0^{+\infty}\|\mathcal G(t,s)\|\,ds
<+\infty.
\end{equation*}

\item\label{H5} For every compact set $\mathcal K\subset\mathbb R^+\times\mathbb R^n$ and every $R>0$, there exists a function $m_{\mathcal K,R}\in L^1(\mathbb R^+)$ such that
\begin{equation}\label{Green-growth-x}
\sup_{(\tau,\xi)\in\mathcal K}
\sup_{t\geq0}
\|\mathcal G(t,s)\|
\omega\bigl(s,|x(s,\tau,\xi)|+R\bigr)
\leq
m_{\mathcal K,R}(s)
\end{equation}
for almost every $s\geq0$.

\item\label{H6} For every compact set $\mathcal K\subset\mathbb R^+\times\mathbb R^n$, there exists a function $n_{\mathcal K}\in L^1(\mathbb R^+)$ such that
\begin{equation}\label{Green-growth-y}
\sup_{(\tau,\eta)\in\mathcal K}
\sup_{t\geq0}
\|\mathcal G(t,s)\|
\omega\bigl(s,|y(s,\tau,\eta)|\bigr)
\leq
n_{\mathcal K}(s)
\end{equation}
for almost every $s\geq0$.

\item\label{H7} The condition
\begin{equation}\label{smallness}
q:=C\gamma<1
\end{equation}
is satisfied.

\end{enumerate}
\end{assumption}

\begin{remark}
Conditions \ref{H5} and \ref{H6} are locally uniform assumptions on the tails of the Green integrals. A bound of the form
$$
\int_0^{+\infty}\|\mathcal G(t,s)\|
\omega\bigl(s,|y(s,\tau,\eta)|\bigr)\,ds<+\infty
$$
for each fixed $(t,\tau,\eta)$ is not sufficient to apply the dominated convergence theorem to a sequence of initial conditions. The locally uniform majorants in \eqref{Green-growth-x} and \eqref{Green-growth-y} are precisely what will be used in the continuity argument.
\end{remark}

The following proposition is a standard result in the theory of differential equations concerning the continuous dependence of solutions on initial conditions. The proof for this proposition is inspired by \cite[Proposition 2]{FLIN}.

\begin{proposition}
The solutions of system \eqref{lin} satisfy
\begin{equation*}
|\xi-\overline\xi|e^{-\alpha|t-\tau|}
\leq
|x(t,\tau,\xi)-x(t,\tau,\overline\xi)|
\leq
|\xi-\overline\xi|e^{\alpha|t-\tau|}.
\end{equation*}
In particular,
$$
|\xi|e^{-\alpha|t-\tau|}
\leq
|x(t,\tau,\xi)|
\leq
|\xi|e^{\alpha|t-\tau|}.
$$
\end{proposition}

\begin{proof}
Since
$$
x(t,\tau,\xi)-x(t,\tau,\overline\xi)
=
T(t,\tau)(\xi-\overline\xi),
$$
the upper estimate follows from \eqref{acotamiento-A} and Gronwall's inequality:
$$
\|T(t,\tau)\|\leq e^{\alpha|t-\tau|}.
$$
On the other hand,
$$
\xi-\overline\xi
=
T(\tau,t)\bigl(x(t,\tau,\xi)-x(t,\tau,\overline\xi)\bigr),
$$
and therefore
$$
|\xi-\overline\xi|
\leq
e^{\alpha|t-\tau|}
|x(t,\tau,\xi)-x(t,\tau,\overline\xi)|.
$$
This proves the lower estimate. The last assertion follows by taking $\overline\xi=0$.
\end{proof}

\vspace{0.5cm}

The following proposition ensures the continuity of the solutions $y(t,\tau,\eta)$ of the nonlinear system (\ref{nolin}) with respect to $(t,\tau,\eta)\in \R^+\times \R^+\times \R^n$. The proof of this proposition  can be found on \cite[Th.~7.1, p.~22]{Coddington}.

\begin{proposition}\label{global-y}
Assume \eqref{acotamiento-A} and \ref{H2}. For every $(\tau,\eta)\in\mathbb R^+\times\mathbb R^n$, the solution $y(t,\tau,\eta)$ of \eqref{nolin} is defined for every $t\geq0$. Moreover, the map
$$
(t,\tau,\eta)\mapsto y(t,\tau,\eta)
$$
is continuous on $\mathbb R^+\times\mathbb R^+\times\mathbb R^n$.
\end{proposition}

\section{Topological equivalence}
\vspace{0.5cm}

The notion of topological equivalence between systems (\ref{lin}) and (\ref{nolin}) was introduced by Palmer in \cite{Palmer}. In broad terms, this concept requires the existence of a homeomorphism mapping solutions of the linear system to solutions of the nonlinear system and vice versa. In the present article, we adopt a weaker version of this notion.

\begin{definition}
The systems \eqref{lin} and \eqref{nolin} are said to be $\mathbb R^+$--continuously topologically equivalent if there exists a function
$$
H:\mathbb R^+\times\mathbb R^n\to\mathbb R^n
$$
such that:
\begin{itemize}
\item[(i)] if $x(t)$ is a solution of \eqref{lin}, then $H(t,x(t))$ is a solution of \eqref{nolin};
\item[(ii)] for each fixed $t\geq0$, $|H(t,\xi)|\to+\infty$ as $|\xi|\to+\infty$;
\item[(iii)] for each fixed $t\geq0$, the map $\xi\mapsto H(t,\xi)$ is a homeomorphism of $\mathbb R^n$;
\item[(iv)] the map $H$ is continuous on $\mathbb R^+\times\mathbb R^n$.
\end{itemize}
The inverse map
$$
G(t,\cdot)=H(t,\cdot)^{-1}
$$
is also required to be continuous on $\mathbb R^+\times\mathbb R^n$, to map solutions of \eqref{nolin} into solutions of \eqref{lin}, and to satisfy $|G(t,\eta)|\to+\infty$ as $|\eta|\to+\infty$ for every fixed $t\geq0$.
\end{definition}

Before stating the main theorem, we construct the maps $H$ and $G$. Let
$$
BC(\mathbb R^+,\mathbb R^n)
=
\left\{
\phi:\mathbb R^+\to\mathbb R^n:
\phi\text{ is continuous and }
\|\phi\|_\infty<+\infty
\right\}.
$$

By Proposition \ref{global-y}, the solution $s\mapsto y(s,\tau,\eta)$ is defined for every $s\geq0$. Therefore, for $(\tau,\eta)\in\mathbb R^+\times\mathbb R^n$, define
\begin{equation}\label{w-star}
w^*(t;(\tau,\eta))
=
-\int_0^{+\infty}
\mathcal G(t,s)f(s,y(s,\tau,\eta))\,ds.
\end{equation}
Condition \ref{H6} implies that $w^*(\cdot;(\tau,\eta))$ belongs to $BC(\mathbb R^+,\mathbb R^n)$.

For $(\tau,\xi)\in\mathbb R^+\times\mathbb R^n$, define the operator
$$
\Gamma_{(\tau,\xi)}:
BC(\mathbb R^+,\mathbb R^n)
\to
BC(\mathbb R^+,\mathbb R^n)
$$
by
\begin{equation*}
(\Gamma_{(\tau,\xi)}\phi)(t)
=
\int_0^{+\infty}
\mathcal G(t,s)
f(s,x(s,\tau,\xi)+\phi(s))\,ds.
\end{equation*}
Condition \ref{H5} shows that this operator is well defined. Moreover, for $\phi,\psi\in BC(\mathbb R^+,\mathbb R^n)$,
\begin{align*}
\|\Gamma_{(\tau,\xi)}\phi-\Gamma_{(\tau,\xi)}\psi\|_\infty
&\leq
\gamma
\sup_{t\geq0}
\int_0^{+\infty}
\|\mathcal G(t,s)\|
|\phi(s)-\psi(s)|\,ds\\
&\leq
C\gamma\|\phi-\psi\|_\infty
=
q\|\phi-\psi\|_\infty.
\end{align*}
Since $q<1$, the operator $\Gamma_{(\tau,\xi)}$ has a unique fixed point $z^*(\cdot;(\tau,\xi))\in BC(\mathbb R^+,\mathbb R^n)$, given by
\begin{equation}\label{z-star}
z^*(t;(\tau,\xi))
=
\int_0^{+\infty}
\mathcal G(t,s)
f(s,x(s,\tau,\xi)+z^*(s;(\tau,\xi)))\,ds.
\end{equation}

Splitting the Green integrals at $s=t$ and using the variation of parameters formula, we obtain
\begin{equation}\label{eq-z}
\frac{d}{dt}z^*(t;(\tau,\xi))
=
A(t)z^*(t;(\tau,\xi))
+
f(t,x(t,\tau,\xi)+z^*(t;(\tau,\xi)))
\end{equation}
and
\begin{equation}\label{eq-w}
\frac{d}{dt}w^*(t;(\tau,\eta))
=
A(t)w^*(t;(\tau,\eta))
-
f(t,y(t,\tau,\eta)).
\end{equation}
By Proposition \ref{global-y}, the nonlinear solutions appearing below are defined on $\mathbb R^+$. By uniqueness of the fixed point and uniqueness of solutions, the following consistency identities hold:
\begin{align}
z^*(r;(\tau,\xi))
&=
z^*(r;(t,x(t,\tau,\xi))),
\label{cons-z}\\
w^*(r;(\tau,\eta))
&=
w^*(r;(t,y(t,\tau,\eta)))
\label{cons-w}
\end{align}
for every $r,t,\tau\geq0$.

We now define
\begin{equation*}
H(t,\xi)
=
\xi+z^*(t;(t,\xi))
\end{equation*}
and
\begin{equation}\label{Homeo-G}
G(t,\eta)
=
\eta+w^*(t;(t,\eta)).
\end{equation}

\begin{theorem}\label{teorema1}
Assume that conditions \ref{H1}--\ref{H7} are satisfied. Then systems \eqref{lin} and \eqref{nolin} are $\mathbb R^+$--continuously topologically equivalent.
\end{theorem}

\begin{proof}
Let $x(t)=x(t,\tau,\xi)$ be a solution of \eqref{lin}. From \eqref{cons-z},
$$
H(t,x(t,\tau,\xi))
=
x(t,\tau,\xi)+z^*(t;(\tau,\xi)).
$$
Equations \eqref{lin} and \eqref{eq-z} show that the right-hand side is a solution of \eqref{nolin}. Similarly, if $y(t)=y(t,\tau,\eta)$ is a solution of \eqref{nolin}, then
$$
G(t,y(t,\tau,\eta))
=
y(t,\tau,\eta)+w^*(t;(\tau,\eta)),
$$
and \eqref{nolin} together with \eqref{eq-w} shows that this function is a solution of \eqref{lin}.

We next prove the continuity of $H$. Let $(t_n,\xi_n)\to(t,\xi)$ and write
$$
z_n=z^*(\cdot;(t_n,\xi_n)),
\qquad
z=z^*(\cdot;(t,\xi)).
$$
The fixed point identities and \eqref{smallness} give
\begin{align*}
\|z_n-z\|_\infty
&\leq
\|\Gamma_{(t_n,\xi_n)}z_n-\Gamma_{(t_n,\xi_n)}z\|_\infty
+
\|\Gamma_{(t_n,\xi_n)}z-\Gamma_{(t,\xi)}z\|_\infty\\
&\leq
q\|z_n-z\|_\infty
+
\|\Gamma_{(t_n,\xi_n)}z-\Gamma_{(t,\xi)}z\|_\infty.
\end{align*}
Therefore,
\begin{equation}\label{fixed-cont}
\|z_n-z\|_\infty
\leq
\frac{1}{1-q}
\|\Gamma_{(t_n,\xi_n)}z-\Gamma_{(t,\xi)}z\|_\infty.
\end{equation}
Choose a compact set $\mathcal K$ containing $(t_n,\xi_n)$ and
$(t,\xi)$, and take $R=\|z\|_\infty$. The continuous dependence of
the linear solutions gives
\begin{equation*}
f(s,x(s,t_n,\xi_n)+z(s))
\longrightarrow
f(s,x(s,t,\xi)+z(s))
\end{equation*}
for every $s\geq0$. Moreover, by condition \ref{H5},
\begin{align*}
&\sup_{r\geq0}\|\mathcal G(r,s)\|
\bigl|
f(s,x(s,t_n,\xi_n)+z(s))
-
f(s,x(s,t,\xi)+z(s))
\bigr|\\
&\qquad\leq
2m_{\mathcal K,R}(s)
\end{align*}
for almost every $s\geq0$. Notice that the first variable of
$\mathcal G$ is $r$, and not $t_n$, since we are estimating the
$BC(\mathbb R^+,\mathbb R^n)$--norm of
$\Gamma_{(t_n,\xi_n)}z-\Gamma_{(t,\xi)}z$. Hence,
\begin{align*}
\|\Gamma_{(t_n,\xi_n)}z-\Gamma_{(t,\xi)}z\|_\infty
&\leq
\int_0^{+\infty}
\sup_{r\geq0}\|\mathcal G(r,s)\|\\
&\quad\cdot
\bigl|
f(s,x(s,t_n,\xi_n)+z(s))
-
f(s,x(s,t,\xi)+z(s))
\bigr|\,ds
\longrightarrow0
\end{align*}
by the dominated convergence theorem.
It follows from \eqref{fixed-cont} that $\|z_n-z\|_\infty\to0$. Since $z$ is continuous,
\begin{align*}
|H(t_n,\xi_n)-H(t,\xi)|
&\leq
|\xi_n-\xi|
+
|z_n(t_n)-z(t)|\\
&\leq
|\xi_n-\xi|
+
\|z_n-z\|_\infty
+
|z(t_n)-z(t)|
\longrightarrow0.
\end{align*}
Thus, $H$ is continuous.

The continuity of $G$ follows in the same manner. Indeed, if $(t_n,\eta_n)\to(t,\eta)$, Proposition \ref{global-y}, condition \ref{H6}, and the dominated convergence theorem imply
$$
\|w^*(\cdot;(t_n,\eta_n))-w^*(\cdot;(t,\eta))\|_\infty
\longrightarrow0.
$$
Consequently,
$$
G(t_n,\eta_n)\longrightarrow G(t,\eta).
$$

It remains to prove that $H(t,\cdot)$ and $G(t,\cdot)$ are inverse maps. Fix $(t,\eta)$ and set
$$
y(s)=y(s,t,\eta),
\qquad
w(s)=w^*(s;(t,\eta)).
$$
By \eqref{eq-w}, the function $x(s)=y(s)+w(s)$ is a solution of \eqref{lin} and
$$
x(t)=\eta+w(t)=G(t,\eta).
$$
Furthermore, from \eqref{w-star},
\begin{align*}
-w(r)
&=
\int_0^{+\infty}
\mathcal G(r,s)f(s,y(s))\,ds\\
&=
\int_0^{+\infty}
\mathcal G(r,s)f(s,x(s)-w(s))\,ds.
\end{align*}

The choice of $-w$ is not arbitrary. Indeed, by the definitions of $G$ and $w$, we have
\begin{equation*}
x(r,t,G(t,\eta))-w(r;(t,\eta))=y(r,t,\eta).
\end{equation*}
Therefore,
\begin{align*}
\left(\Gamma_{(t,G(t,\eta))}(-w)\right)(r)
&=
\int_0^{+\infty}
\mathcal G(r,s)
f\left(s,x(s,t,G(t,\eta))-w(s;(t,\eta))\right)ds\\
&=
\int_0^{+\infty}
\mathcal G(r,s)
f(s,y(s,t,\eta))ds\\
&=
-w(r;(t,\eta)).
\end{align*}
Hence, $-w$ is a fixed point of $\Gamma_{(t,G(t,\eta))}$. By the uniqueness of the fixed point, we obtain
\begin{equation*}
z^*(r;(t,G(t,\eta)))
=
-w(r;(t,\eta)).
\end{equation*}
Consequently,
\begin{align*}
H(t,G(t,\eta))
&=
G(t,\eta)
+
z^*(t;(t,G(t,\eta)))\
&=
\eta+w(t;(t,\eta))-w(t;(t,\eta))\
&=
\eta.
\end{align*}

Hence, $-w$ is a fixed point of $\Gamma_{(t,G(t,\eta))}$. By uniqueness of the fixed point,
$$
z^*(r;(t,G(t,\eta)))=-w(r).
$$
Taking $r=t$, we obtain
$$
H(t,G(t,\eta))
=
G(t,\eta)-w(t)
=
\eta.
$$

Conversely, fix $(t,\xi)$, set
$$
x(s)=x(s,t,\xi),
\qquad
z(s)=z^*(s;(t,\xi)),
$$
and define $y(s)=x(s)+z(s)$. By \eqref{eq-z}, $y$ is a solution of \eqref{nolin} and $y(t)=H(t,\xi)$. Equation \eqref{z-star} shows that $-z$ coincides with $w^*(\cdot;(t,H(t,\xi)))$. Therefore,
$$
G(t,H(t,\xi))=\xi.
$$
Thus, for every fixed $t\geq0$, the maps $H(t,\cdot)$ and $G(t,\cdot)$ are mutually inverse homeomorphisms.

Finally, we verify the behavior at infinity. Suppose that $|\xi_n|\to+\infty$ but $\{H(t,\xi_n)\}$ is bounded. Passing to a subsequence, we may assume that $H(t,\xi_n)\to\eta$. Since $G(t,\cdot)$ is continuous,
$$
\xi_n
=
G(t,H(t,\xi_n))
\longrightarrow
G(t,\eta),
$$
which is a contradiction. Hence, $|H(t,\xi)|\to+\infty$ as $|\xi|\to+\infty$. The corresponding statement for $G$ follows by the same argument. This completes the proof.
\end{proof}

\section{Differentiability of the topological equivalence}

In this section, we establish the differentiability of the topological equivalence obtained in Theorem \ref{teorema1}. The main point is that the required estimates must be imposed on the Green operator acting on the variational equations. No boundedness of the nonlinear solutions on $\mathbb R^+$ will be used.

Fix an integer $r\geq1$. For $(t,\eta)\in\mathbb R^+\times\mathbb R^n$, set
$$
Y_j(s;t,\eta):=D_\eta^j y(s,t,\eta),
\qquad 1\leq j\leq r.
$$
Thus, $Y_j(s;t,\eta)$ is a symmetric $j$--linear map. Whenever $f(t,\cdot)$ is of class $C^r$, define
$$
\mathcal F_j(s;t,\eta)
:=
D_\eta^j\bigl[f(s,y(s,t,\eta))\bigr].
$$
In particular,
\begin{align*}
\mathcal F_1(s;t,\eta)
&=
D_yf(s,y(s,t,\eta))Y_1(s;t,\eta),\\
\mathcal F_2(s;t,\eta)
&=
D_y^2f(s,y(s,t,\eta))[Y_1(s;t,\eta),Y_1(s;t,\eta)]\\
&\quad +D_yf(s,y(s,t,\eta))Y_2(s;t,\eta).
\end{align*}
More generally, the Fa\`a di Bruno formula gives
\begin{equation}\label{F-j}
\mathcal F_j(s;t,\eta)
=
\sum_{k=1}^j
D_y^k f(s,y(s,t,\eta))
\mathfrak B_{j,k}
\bigl(Y_1(s;t,\eta),\ldots,Y_{j-k+1}(s;t,\eta)\bigr),
\end{equation}
where $\mathfrak B_{j,k}$ denotes the corresponding multilinear Bell polynomial.

\begin{assumption}\label{hyp-regularity}
The following conditions hold.
\begin{enumerate}[label=\textnormal{(R\arabic*)},leftmargin=1.4cm]

\item\label{R1} For every $t\geq0$, the function $f(t,\cdot)$ is of class $C^r$, and the derivatives
$$
(t,y)\mapsto D_y^j f(t,y),
\qquad 1\leq j\leq r,
$$
are continuous on $\mathbb R^+\times\mathbb R^n$.

\item\label{R2} For every compact set $\mathcal K\subset\mathbb R^+\times\mathbb R^n$ and every $1\leq j\leq r$, there exists a function $\ell_{\mathcal K,j}\in L^1(\mathbb R^+)$ such that
\begin{equation}\label{regularity-majorant}
\sup_{(t,\eta)\in\mathcal K}
\left\|
\mathcal G(t,s)\mathcal F_j(s;t,\eta)
\right\|
\leq
\ell_{\mathcal K,j}(s)
\end{equation}
for almost every $s\geq0$.

\item\label{R3} There exists $\vartheta\in(0,1)$ such that
\begin{equation*}
\sup_{(t,\eta)\in\mathbb R^+\times\mathbb R^n}
\int_0^{+\infty}
\left\|
\mathcal G(t,s)
D_yf(s,y(s,t,\eta))Y_1(s;t,\eta)
\right\|\,ds
\leq\vartheta.
\end{equation*}

\end{enumerate}
\end{assumption}

\begin{remark}
Condition \ref{R2} is a Green-integrability condition for the variational equations. It does not require the functions $s\mapsto y(s,t,\eta)$ or $s\mapsto Y_j(s;t,\eta)$ to be bounded. Their possible growth is allowed, provided it is compensated by the Green operator and by the time decay of the derivatives of $f$.

A directly verifiable sufficient condition for \ref{R2} is the following. We denote by $\pi_1(\mathcal K)$ the projection of $\mathcal K$ onto the time variable. For every compact set $\mathcal K$ and every $1\leq j\leq r$, assume that there exist nonnegative functions $a_{\mathcal K,k}$ and $b_{\mathcal K,m}$ such that
$$
\sup_{(t,\eta)\in\mathcal K}
\|D_y^k f(s,y(s,t,\eta))\|
\leq a_{\mathcal K,k}(s),
$$
$$
\sup_{(t,\eta)\in\mathcal K}
\|Y_m(s;t,\eta)\|
\leq b_{\mathcal K,m}(s),
$$
and
$$
\sup_{t\in\pi_1(\mathcal K)}\|\mathcal G(t,s)\|
\sum_{k=1}^j
a_{\mathcal K,k}(s)
\mathfrak B_{j,k}
\bigl(b_{\mathcal K,1}(s),\ldots,b_{\mathcal K,j-k+1}(s)\bigr)
\in L^1(\mathbb R^+).
$$
Then \eqref{regularity-majorant} follows from \eqref{F-j}.
\end{remark}

\begin{remark}
Condition \ref{R3} is only used to guarantee that $D_\eta G(t,\eta)$ is invertible at every point. It can be replaced by the weaker assumption
$$
\det D_\eta G(t,\eta)\neq0,
\qquad
(t,\eta)\in\mathbb R^+\times\mathbb R^n,
$$
whenever this property can be verified directly.
\end{remark}

\begin{definition}
The systems \eqref{lin} and \eqref{nolin} are said to be $C^r$--continuously topologically equivalent on $\mathbb R^+$ if they are $\mathbb R^+$--continuously topologically equivalent and the following conditions hold:
\begin{itemize}
\item[(i)] for each fixed $t\geq0$, the map $\xi\mapsto H(t,\xi)$ is a $C^r$--diffeomorphism of $\mathbb R^n$;
\item[(ii)] the derivatives of $H$ and $G$ with respect to the space variable, up to order $r$, are continuous functions on $\mathbb R^+\times\mathbb R^n$.
\end{itemize}
\end{definition}

\begin{theorem}
Assume that conditions \ref{H1}--\ref{H7} and \ref{R1}--\ref{R3} are satisfied. Then systems \eqref{lin} and \eqref{nolin} are $C^r$--continuously topologically equivalent on $\mathbb R^+$.
\end{theorem}

\begin{proof}
By Theorem \ref{teorema1}, the systems are $\mathbb R^+$--continuously topologically equivalent. We first prove the regularity of $G$. By Proposition \ref{global-y}, the solution $s\mapsto y(s,t,\eta)$ is defined on $\mathbb R^+$ for every $(t,\eta)\in\mathbb R^+\times\mathbb R^n$. From \eqref{Homeo-G},
\begin{equation}\label{G-integral-smooth}
G(t,\eta)
=
\eta-
\int_0^{+\infty}
\mathcal G(t,s)f(s,y(s,t,\eta))\,ds.
\end{equation}
Condition \ref{R1} and the standard differentiability theorem for solutions with respect to initial conditions imply that, for each fixed $s,t\geq0$, the map
$$
\eta\mapsto y(s,t,\eta)
$$
is of class $C^r$. Consequently, the map
$$
\eta\mapsto f(s,y(s,t,\eta))
$$
is of class $C^r$, and its $j$--th derivative is $\mathcal F_j(s;t,\eta)$.

Let $(t_0,\eta_0)\in\mathbb R^+\times\mathbb R^n$. Choose a compact neighborhood $\mathcal K$ of $(t_0,\eta_0)$. By \ref{R2}, every derivative of the integrand in \eqref{G-integral-smooth}, up to order $r$, is bounded on $\mathcal K$ by an integrable function. Therefore, differentiation under the integral sign gives
\begin{equation}\label{DG-new}
D_\eta G(t,\eta)
=
I_n-
\int_0^{+\infty}
\mathcal G(t,s)\mathcal F_1(s;t,\eta)\,ds
\end{equation}
and, for $2\leq j\leq r$,
\begin{equation}\label{DjG-new}
D_\eta^jG(t,\eta)
=
-
\int_0^{+\infty}
\mathcal G(t,s)\mathcal F_j(s;t,\eta)\,ds.
\end{equation}
Thus, for each fixed $t\geq0$, the map $\eta\mapsto G(t,\eta)$ is of class $C^r$.

We now prove that the derivatives in \eqref{DG-new}--\eqref{DjG-new} are continuous with respect to $(t,\eta)$. Let $(t_m,\eta_m)\to(t,\eta)$ and choose a compact set $\mathcal K$ containing this sequence and its limit. For every $s\neq t$, the continuity of the transition matrix and of the projections implies
$$
\mathcal G(t_m,s)\longrightarrow\mathcal G(t,s).
$$
The possible discontinuity at $s=t$ occurs at a set of measure zero. Moreover, the continuous dependence of the variational equations gives
$$
\mathcal F_j(s;t_m,\eta_m)
\longrightarrow
\mathcal F_j(s;t,\eta)
$$
for every fixed $s\geq0$. Condition \ref{R2} and the dominated convergence theorem yield
$$
D_\eta^jG(t_m,\eta_m)
\longrightarrow
D_\eta^jG(t,\eta),
\qquad 1\leq j\leq r.
$$
Hence, the derivatives of $G$ with respect to $\eta$, up to order $r$, are continuous functions of $(t,\eta)$.

From \eqref{DG-new} and \ref{R3},
$$
\|D_\eta G(t,\eta)-I_n\|
\leq\vartheta<1.
$$
Therefore, $D_\eta G(t,\eta)$ is invertible for every $(t,\eta)$, and
$$
\|D_\eta G(t,\eta)^{-1}\|
\leq\frac{1}{1-\vartheta}.
$$
For each fixed $t\geq0$, the map $G(t,\cdot)$ is already a homeomorphism by Theorem \ref{teorema1}. Since it is also a local $C^r$--diffeomorphism at every point, its inverse $H(t,\cdot)$ is a global $C^r$--diffeomorphism.

It remains to prove the joint continuity of the derivatives of $H$. Differentiating
$$
G(t,H(t,\xi))=\xi
$$
with respect to $\xi$, we obtain
\begin{equation*}
D_\xi H(t,\xi)
=
\left[D_\eta G(t,H(t,\xi))\right]^{-1}.
\end{equation*}
Since $H$, $D_\eta G$ and the inversion map on the set of invertible matrices are continuous, $D_\xi H$ is continuous on $\mathbb R^+\times\mathbb R^n$.

For $2\leq j\leq r$, differentiating the identity $G(t,H(t,\xi))=\xi$ $j$ times, the term containing the highest derivative of $H$ is
$$
D_\eta G(t,H(t,\xi))D_\xi^jH(t,\xi),
$$
and all the remaining terms depend only on derivatives of $G$ of order at most $j$ and derivatives of $H$ of order at most $j-1$. Since $D_\eta G(t,H(t,\xi))$ is invertible, $D_\xi^jH(t,\xi)$ can be solved recursively. An induction on $j$, together with the continuity already established for the derivatives of $G$, shows that all derivatives $D_\xi^jH$, $1\leq j\leq r$, are continuous on $\mathbb R^+\times\mathbb R^n$. This proves the result.
\end{proof}

\section{A class of examples}

The following example shows that the hypotheses allow perturbations which are unbounded with respect to the space variable and, at the same time, produce a $C^r$--continuous topological equivalence. Let $a,b,\beta,\varepsilon>0$ and consider
\begin{equation}\label{example-linear}
A=
\begin{pmatrix}
-a&0\\
0&b
\end{pmatrix}
\end{equation}
and
\begin{equation}\label{example-f}
f(t,y)
=
\varepsilon e^{-\beta t}
\begin{pmatrix}
y_1+\sin y_1\\
y_2+\sin y_2
\end{pmatrix}.
\end{equation}
The transition matrix is
$$
T(t,s)
=
\begin{pmatrix}
e^{-a(t-s)}&0\\
0&e^{b(t-s)}
\end{pmatrix}.
$$
Taking
$$
P=
\begin{pmatrix}
1&0\\
0&0
\end{pmatrix},
\qquad
Q=
\begin{pmatrix}
0&0\\
0&1
\end{pmatrix},
$$
system \eqref{example-linear} admits a $(\mu,\nu)$--dichotomy with
$$
\mu(t)=e^t,
\qquad
\nu(t)=1,
\qquad
M=1,
\qquad
\lambda=\min\{a,b\}.
$$
Moreover,
$$
\sup_{t\geq0}
\int_0^{+\infty}\|\mathcal G(t,s)\|\,ds
\leq
\frac{1}{a}+\frac{1}{b}.
$$
The perturbation \eqref{example-f} is unbounded with respect to $y$ for every fixed $t\geq0$, and it satisfies
$$
|f(t,y)-f(t,\overline y)|
\leq
2\varepsilon|y-\overline y|.
$$
Thus, we may take
$$
\gamma=2\varepsilon
\qquad\text{and}\qquad
\omega(t,\abs{y})=\varepsilon e^{-\beta t}(\abs{y}+\sqrt{2}).
$$

Set $\rho=\max\{a,b\}$. For every compact set of initial conditions, there exists a constant $L>0$ such that the corresponding linear and nonlinear solutions satisfy
$$
|x(s,\tau,\xi)|+|y(s,\tau,\eta)|
\leq
Le^{\rho s},
\qquad s\geq0.
$$
The estimate for the nonlinear solution follows from Gronwall's inequality and the integrability of $e^{-\beta s}$. Hence, if $\beta>\rho$, then conditions \ref{H5} and \ref{H6} hold. Therefore, the topological equivalence follows whenever
\begin{equation}\label{example-top-small}
2\varepsilon\left(\frac{1}{a}+\frac{1}{b}\right)<1.
\end{equation}

We now verify the differentiability hypotheses. Since the system is decoupled, the mixed derivatives of the solutions with respect to the initial conditions vanish. Moreover, all derivatives of $y_i+\sin y_i$ of order at least one are bounded, and
$$
\|D_yf(s,y)\|
\leq2\varepsilon e^{-\beta s},
\qquad
\|D_y^k f(s,y)\|
\leq\varepsilon e^{-\beta s},
\quad k\geq2.
$$
For the first variational equation, on the support of each component of the Green function, the exponential factor of the linear transition matrix cancels with the corresponding factor of $\mathcal G(t,s)$. Consequently,
$$
\left\|
\mathcal G(t,s)
D_yf(s,y(s,t,\eta))Y_1(s;t,\eta)
\right\|
\leq
2\varepsilon e^{2\varepsilon/\beta}e^{-\beta s}.
$$
Thus, condition \ref{R3} holds if
\begin{equation}\label{example-diff-small}
\frac{2\varepsilon}{\beta}e^{2\varepsilon/\beta}<1.
\end{equation}

For the higher-order variational equations, an induction on $j$ shows that, for every compact set $\mathcal K\subset\mathbb R^+\times\mathbb R^n$, there exists $C_{\mathcal K,j}>0$ such that the unstable component satisfies
$$
\|Y_j(s;t,\eta)\|
\leq
C_{\mathcal K,j}e^{jb(s-t)},
\qquad
s\geq t,
\quad
(t,\eta)\in\mathcal K.
$$
The stable component only contributes when $0\leq s\leq t$, which is a bounded interval when $(t,\eta)$ varies in a compact set. Using \eqref{F-j}, for $s\geq t$ we obtain
$$
\|\mathcal G(t,s)\mathcal F_j(s;t,\eta)\|
\leq
C_{\mathcal K,j}
e^{-\beta s}e^{(j-1)b(s-t)}.
$$
Therefore, condition \ref{R2} holds for every $1\leq j\leq r$ provided
\begin{equation*}
\beta>(r-1)b.
\end{equation*}

We conclude that, if
$$
\beta>\max\{\rho,(r-1)b\},
$$
and conditions \eqref{example-top-small} and \eqref{example-diff-small} are satisfied, then systems \eqref{example-linear} and \eqref{nolin}, with the perturbation given by \eqref{example-f}, are $C^r$--continuously topologically equivalent on $\mathbb R^+$.

\begin{example}
The following example considers a nonautonomous linear system with a growth rate containing both exponential and polynomial factors. Fix $a,b,\kappa,\beta,\varepsilon>0$ and define
$$
h(t):=1+\frac{\kappa}{1+t}.
$$
Consider the linear system
\begin{equation}\label{example-mixed-linear}
x'=A(t)x,
\qquad
A(t)=
\begin{pmatrix}
-ah(t)&0\\
0&bh(t)
\end{pmatrix},
\end{equation}
together with the perturbation
\begin{equation}\label{example-mixed-f}
f(t,y)
=
\varepsilon e^{-\beta t}
\begin{pmatrix}
y_1+\sin y_1\\
y_2+\sin y_2
\end{pmatrix}.
\end{equation}
Notice that $A$ is continuous and uniformly bounded, since
$$
\sup_{t\geq0}\|A(t)\|
\leq
(1+\kappa)\max\{a,b\}.
$$

Set
$$
\mu(t):=e^t(1+t)^\kappa.
$$
Since
$$
\frac{\mu'(t)}{\mu(t)}
=
1+\frac{\kappa}{1+t}
=
h(t),
$$
the transition matrix of system \eqref{example-mixed-linear} is
$$
T(t,s)
=
\begin{pmatrix}
\left[\dfrac{\mu(t)}{\mu(s)}\right]^{-a}&0\\[2mm]
0&
\left[\dfrac{\mu(t)}{\mu(s)}\right]^b
\end{pmatrix}.
$$
Taking
$$
P=
\begin{pmatrix}
1&0\\
0&0
\end{pmatrix},
\qquad
Q=
\begin{pmatrix}
0&0\\
0&1
\end{pmatrix},
$$
system \eqref{example-mixed-linear} admits a $(\mu,\nu)$--dichotomy with
$$
\nu(t)=1,
\qquad
M=1,
\qquad
\lambda=\min\{a,b\}.
$$
This dichotomy is generated by the mixed exponential--polynomial growth rate
$$
\mu(t)=e^t(1+t)^\kappa,
$$
and therefore it is different from the purely exponential rate considered in the previous example.

The corresponding Green function is given by
$$
\mathcal G(t,s)
=
\begin{cases}
\begin{pmatrix}
\left[\dfrac{\mu(t)}{\mu(s)}\right]^{-a}&0\\
0&0
\end{pmatrix},
&0\leq s\leq t,\\[5mm]
-
\begin{pmatrix}
0&0\\
0&
\left[\dfrac{\mu(t)}{\mu(s)}\right]^b
\end{pmatrix},
&0\leq t<s.
\end{cases}
$$
Since
$$
\frac{\mu(t)}{\mu(s)}
=
e^{t-s}
\left(\frac{1+t}{1+s}\right)^\kappa,
$$
we have
$$
\left[\frac{\mu(t)}{\mu(s)}\right]^{-a}
\leq e^{-a(t-s)},
\qquad
t\geq s,
$$
and
$$
\left[\frac{\mu(t)}{\mu(s)}\right]^b
\leq e^{-b(s-t)},
\qquad
t<s.
$$
Consequently,
\begin{align*}
\sup_{t\geq0}
\int_0^{+\infty}\|\mathcal G(t,s)\|\,ds
&\leq
\sup_{t\geq0}
\left(
\int_0^t e^{-a(t-s)}\,ds
+
\int_t^{+\infty}e^{-b(s-t)}\,ds
\right)\\
&\leq
\frac{1}{a}+\frac{1}{b}.
\end{align*}
Hence, condition \ref{H4} is satisfied with
$$
C\leq\frac{1}{a}+\frac{1}{b}.
$$

The perturbation \eqref{example-mixed-f} is unbounded with respect to the space variable for every fixed $t\geq0$. Moreover,
$$
|f(t,y)-f(t,\overline y)|
\leq
2\varepsilon e^{-\beta t}|y-\overline y|
\leq
2\varepsilon|y-\overline y|.
$$
Therefore, condition \ref{H2} holds with
$$
\gamma=2\varepsilon.
$$
Condition \ref{H3} holds with
$$
\omega(t,r)
=
\varepsilon e^{-\beta t}(r+\sqrt{2}).
$$

We next verify conditions \ref{H5} and \ref{H6}. For every compact set of initial conditions, the variation of constants formula and Gronwall's inequality imply that there exists $L>0$ such that
\begin{equation}\label{mixed-solution-bound}
|x(s,\tau,\xi)|+|y(s,\tau,\eta)|
\leq
L\left(1+\mu(s)^b\right),
\qquad
s\geq0.
\end{equation}
Indeed, the perturbation contributes the bounded factor
$$
\exp\left(
2\varepsilon\int_0^{+\infty}e^{-\beta u}\,du
\right)
=
e^{2\varepsilon/\beta},
$$
while the possible forward growth is determined by the unstable component of the linear system.

On the other hand,
$$
\sup_{t\geq0}\|\mathcal G(t,s)\|\leq1.
$$
Using \eqref{mixed-solution-bound}, we obtain
\begin{align*}
\|\mathcal G(t,s)\|
\omega\bigl(s,|x(s,\tau,\xi)|+R\bigr)
&\leq
C_1e^{-\beta s}
\left(1+\mu(s)^b\right)\\
&=
C_1e^{-\beta s}
\left(1+e^{bs}(1+s)^{b\kappa}\right),
\end{align*}
and an analogous estimate holds along the nonlinear solutions. Therefore, conditions \ref{H5} and \ref{H6} are satisfied whenever
\begin{equation*}
\beta>b.
\end{equation*}

The topological smallness condition \ref{H7} follows if
\begin{equation}\label{mixed-top-small}
2\varepsilon
\left(
\frac{1}{a}+\frac{1}{b}
\right)
<1.
\end{equation}

We now verify the differentiability assumptions. The function $f(t,\cdot)$ is of class $C^\infty$, and
$$
\|D_yf(s,y)\|
\leq
2\varepsilon e^{-\beta s},
$$
while
$$
\|D_y^k f(s,y)\|
\leq
\varepsilon e^{-\beta s},
\qquad
k\geq2.
$$
Thus, condition \ref{R1} is satisfied for every $r\geq1$.

Since the system is diagonal, the first variational equation is also diagonal. On the support of each component of the Green function, the corresponding factor of the linear transition matrix cancels with the factor contained in $Y_1(s;t,\eta)$. Furthermore,
$$
\exp\left(
2\varepsilon
\left|
\int_t^s e^{-\beta u}\,du
\right|
\right)
\leq
e^{2\varepsilon/\beta}.
$$
It follows that
$$
\left\|
\mathcal G(t,s)
D_yf(s,y(s,t,\eta))
Y_1(s;t,\eta)
\right\|
\leq
2\varepsilon e^{2\varepsilon/\beta}e^{-\beta s}.
$$
Consequently, condition \ref{R3} is satisfied if
\begin{equation}\label{mixed-diff-small}
\frac{2\varepsilon}{\beta}
e^{2\varepsilon/\beta}
<1.
\end{equation}

For the higher-order variational equations, an induction on $j$ shows that, for every compact set $\mathcal K\subset\mathbb R^+\times\mathbb R^n$, there exists $C_{\mathcal K,j}>0$ such that
$$
\|Y_j(s;t,\eta)\|
\leq
C_{\mathcal K,j}
\left[
\frac{\mu(s)}{\mu(t)}
\right]^{jb},
\qquad
s\geq t,
\quad
(t,\eta)\in\mathcal K.
$$
Using the Fa\`a di Bruno formula, we obtain
\begin{align*}
\|\mathcal G(t,s)\mathcal F_j(s;t,\eta)\|
&\leq
C_{\mathcal K,j}
e^{-\beta s}
\left[
\frac{\mu(s)}{\mu(t)}
\right]^{(j-1)b}\\
&\leq
\widetilde C_{\mathcal K,j}
e^{-[\beta-(j-1)b]s}
(1+s)^{(j-1)b\kappa},
\end{align*}
for $s\geq t$. The stable component is only present when $0\leq s\leq t$, which is a bounded interval when $(t,\eta)$ varies in a compact set. Hence, condition \ref{R2} holds for every $1\leq j\leq r$ provided
\begin{equation*}
\beta>(r-1)b.
\end{equation*}

We conclude that if
$$
\beta>\max\{b,(r-1)b\},
$$
and conditions \eqref{mixed-top-small} and \eqref{mixed-diff-small} are satisfied, then all conditions \ref{H1}--\ref{H7} and \ref{R1}--\ref{R3} hold. Therefore, systems \eqref{example-mixed-linear} and \eqref{nolin}, with the perturbation given by \eqref{example-mixed-f}, are $C^r$--continuously topologically equivalent on $\mathbb R^+$.

For instance, one may take
$$
a=b=\kappa=1,
\qquad
\varepsilon=\frac{1}{16},
\qquad
\beta=r+1.
$$
With this choice, all the previous inequalities are satisfied.
\end{example}

\section{Final remarks}

The main point of the manuscript is that the differentiability of the topological equivalence can be obtained without assuming that the nonlinear perturbation or the nonlinear solutions are bounded. The required control is expressed directly through the Green operator acting on the successive variational equations. This formulation permits the growth of the nonlinear solutions and their derivatives with respect to the initial conditions, provided that such growth is compensated by the decay of the nonlinear terms and by the hyperbolic estimates contained in the Green function.

The example above shows that the assumptions are nonempty and can be checked explicitly for an unbounded perturbation. In particular, the topological smallness condition and the condition guaranteeing the invertibility of $D_\eta G$ play different roles and should be stated separately.

\end{document}